\documentclass[11pt,reqno]{amsart}

\usepackage[T1]{fontenc}
\usepackage[english]{babel}
\usepackage{lmodern}
\usepackage{microtype}
\usepackage{mathtools}
\usepackage{amssymb}
\usepackage[colorlinks=true,linkcolor=blue,citecolor=blue,urlcolor=blue]{hyperref}

\newtheorem{theorem}{Theorem}[section]
\newtheorem{lemma}[theorem]{Lemma}

\newtheorem{corollary}[theorem]{Corollary}
\theoremstyle{definition}
\newtheorem{definition}[theorem]{Definition}

\theoremstyle{remark}
\newtheorem{remark}[theorem]{Remark}

\newcommand{\id}{\operatorname{id}}
\newcommand{\cM}{\mathcal{M}}
\newcommand{\R}{\mathbb{R}}
\newcommand{\Sing}{\operatorname{Sing}}
\newcommand{\htop}{h_{\mathrm{top}}}

\newcommand{\huc}{h^{\mathrm{uc}}_{\mathrm{top}}}
\newcommand{\hB}{h^{\mathrm{B}}_{\mathrm{top}}}
\newcommand{\hE}{e^{\mathrm{uc}}_{\mathrm{top}}}

\title[Entropy and periodic-orbit growth]
{Entropy on Regular Sets and Periodic-Orbit Growth for Singular Flows}

\author{C.~A. Morales}
\address{Hangzhou International Innovation Institute of Beihang University,
Hangzhou 311115, China}
\email{tchivatze@gmail.com}

\subjclass[2020]{Primary 37B40; Secondary 37D35, 37C27}
\keywords{Topological entropy, periodic orbit, flow, singularity,
geometric separation}

\begin{document}

\begin{abstract}
We study entropy and periodic-orbit growth for flows on metric spaces.
First, we prove that the finite topological entropy of a flow on a compact
metric space is completely carried by compact subsets of its regular set.
Next, we establish a Bowen--Walters inequality for geometrically separating
flows on possibly noncompact metric spaces, under uniform control at time
zero and dynamical isolation at infinity.  As applications, we obtain the Bowen--Walters inequality for singular
suspension flows over expansive homeomorphisms,
multisingular-hyperbolic sets---extending the upper-bound part of
\cite{pyyz}---and asymptotically sectional-hyperbolic attractors such
as Rovella's \cite{r}.
\end{abstract}

\maketitle

\section{Introduction}

\noindent
Let $\phi=(\phi_t)_{t\in\mathbb R}$
be a flow on a metric space $X$. This means that the joint map $(x,t)\mapsto \phi_t(x)$ is continuous, $\phi_0=\id_X$ and
$\phi_{t+u}=\phi_t\circ\varphi_u$ for all $t,u\in\R$.
Its singular and regular sets are
\[
 \Sing(\phi)
 :=\{x\in X:\phi_t(x)=x\text{ for every }t\in\R\}
 \quad\text{and}\quad
 X^*:=X\setminus\Sing(\phi),
\]
respectively.  The invariant set $X^*$ need not be
compact even if $X$ is.  Thus, the dynamics of a compact flow with singularities naturally
leads to the study of a flow on a noncompact metric space.

This paper has two main purposes.  First, we show that if $X$ is compact, the topological entropy of $\phi$ is the supremum of the continuous-time
separated/spanning entropies of compact subsets of $X^*$.
More precisely, we prove the following result.

\begin{theorem}[Entropy on the regular set]
\label{thmA}
Let $\phi$ be a flow of a compact metric space $X$.
If $e(\phi)<\infty$, then
\[
 e(\phi)
 =
 \sup_{\substack{K\subset X^*\\K\ \mathrm{compact}}}
 \hE(\phi,K).
\]
\end{theorem}

Here
\[
 e(\phi):=\htop(\phi_1),
\]
where $\htop(f)$ is the topological entropy of a continuous map $f:X\to X$ (\cite{akm}), while $\hE(\phi,K)$ is the upper-capacity entropy defined directly from orbit segments of the
flow (Definition \ref{UC}).
The identity in this theorem was considered earlier in \cite{rwy}.

Our second result is motivated by the Bowen-Walters inequality
$$
\limsup_{T\to\infty}\frac{1}T\log\nu(T)\leq e(\phi),
$$
where $\nu(T)$ is the number of distinct periodic orbits $\gamma$ with period $\pi(\gamma)\in [0,T]$.
Recall that $x\in X$ is periodic if there is a minimal number $\pi(x)>0$ (called period) such that $\phi_{\pi(x)}(x)=x$. We say that $\gamma$ is a periodic orbit if $\gamma=\phi_\mathbb{R}(x)$ for some periodic point $x$. Its period denoted by $\pi(\gamma)$ is that of $x$.

It was proved on compact metric spaces for expansive flows \cite{bw}, $N$-expansive flows \cite{lms} and asymptotically expansive semiflows \cite{lr}.
Our goal is to extend this inequality to the noncompact case.
This requires explicit substitutes for $e(\phi)$ along with some additional hypotheses.
First we assume that $\phi$ is {\em dynamical isolation at infinity}, that is, there is a compact set meeting every regular orbit.  We also assume that $\phi$ is {\em uniformly $C_0$}, that is,
\[
 \lim_{t\to0}\sup_{x\in X}d(x,\phi_t(x))=0.
\]
This term is motivated by the classical notion of uniformly continuous semigroups of linear operators \cite{p}.
We also replace the
expansivity assumption by {\em geometric separation} \cite{a} and $e(\phi)$ by
\[
 \hE(\phi)
 :=\sup_{\substack{K\subset X\\K\ \mathrm{compact}}}
 \hE(\phi,K)
\]
which is the flow-version of the entropy for continuous maps introduced by Bowen and Dinaburg \cite{b1}, \cite{di}.
Under such definitions we state our second result.

\begin{theorem}[Noncompact Bowen--Walters inequality]
\label{thmB}
Let $\phi$ be a geometrically separating, uniformly $C_0$ flow on a metric
space $X$.  If $\phi$ is dynamically isolated at infinity, then
\begin{equation}
\label{pelus}
 \limsup_{T\to\infty}\frac1T\log\nu(T)
 \leq \hE(\phi).
\end{equation}
\end{theorem}

The hypothesis of dynamical isolation at infinity cannot, in general,
be omitted. Indeed, consider the disjoint union of infinitely many
uniformly separated copies of a fixed Anosov flow, embedded in a
Euclidean space. The resulting flow is expansive and satisfies
\[
0<\hE(\phi)<\infty,
\]
but it has infinitely many periodic orbits of period at most $T$ for
every sufficiently large $T$. Thus $\nu(T)=\infty$ for all sufficiently
large $T$, and \eqref{pelus} fails. This construction is given in
Example~1 of \cite{ym1}.
Geometric separation also cannot be replaced by weaker expansivity
conditions such as KH-expansivity or separation, even when the phase
space is compact; see Example~2.10 of \cite{a}. But the
theorem applies to $C$-separating flows \cite{h} since such flows are geometrically separating.
On the other hand, the uniformly $C_0$ hypothesis is essential to prove Theorem \ref{thmB}. Indeed,
it provides uniform control of the spatial displacement caused by the
small time errors arising in the piecewise-linear reparametrizations
used to compare periodic orbits. Without this assumption, points near
infinity may undergo a fixed spatial displacement in arbitrarily short
times, and the periodic-orbit separation argument breaks down. We do
not know whether the conclusion remains valid without this hypothesis. This problem was considered in \cite{ym1}.

The paper is organized as follows.  In Section~\ref{sec2}, we introduce the upper-capacity entropy of compact subsets and compare its continuous-time
definition with the corresponding definition for the time-one map.  In Section \ref{sec4}, we recall the required facts about Bowen entropy of
subsets.  The theorems will be proved in
Section~\ref{sec5}.  Finally, in Section~\ref{sec8}, we apply our results to several
important classes of singular flows. These applications establish the
Bowen--Walters inequality for singular suspensions over expansive
homeomorphisms, for multisingular-hyperbolic sets---thereby extending
the upper-bound result of \cite{pyyz}---and for asymptotically
sectional-hyperbolic attractors, including the Rovella attractor \cite{r}.

\section{Upper Capacity Entropy for Flows}
\label{sec2}

\noindent
In this section, we define the entropy of
a compact set directly in continuous time and compare it with the
corresponding entropy for the time-one map.

Along this section $\varphi$ will denote a flow of a metric space $X$. We do not assume that $X$ is compact unless other stated. We write
$$
\varphi_I(x)=\{\varphi_t(x)\mid t\in I\}
$$
whenever $x\in X$ and $I\subset\mathbb{R}$.

\subsection{Separated and spanning sets for the flow}

For $T>0$, define the dynamical metric
\[
d_T^\varphi(x,y)
:=
\max_{0\leq t\leq T}
d\bigl(\varphi_t(x),\varphi_t(y)\bigr).
\]
Let $K\subset X$ be compact and let $\varepsilon>0$.

\begin{definition}
A set $E\subset K$ is called $(T,\varepsilon)$-separated for the flow
if
\[
d_T^\varphi(x,y)>\varepsilon
\]
for every pair of distinct points $x,y\in E$.
\end{definition}

Define the small-time displacement modulus
\[
 \omega(r)=\sup\bigl\{d(z,\varphi_u(z)):z\in X, |u|\le r\bigr\},
 \qquad r\ge0.
\]
It follows that $\varphi$ is uniformly $C_0$ if and only if
\[
 \lim_{r\downarrow0}\omega(r)=0.
\]
Notice that the flow identity gives
\[
 d\bigl(\varphi_{t+u}(x),\varphi_t(x)\bigr)
 =d\bigl(\varphi_u(\varphi_t(x)),\varphi_t(x)\bigr)\le \omega(|u|).
\]
Thus uniform $C_0$ continuity is precisely uniform equicontinuity in time of
all orbit maps.  It holds automatically when $X$ is compact.
It also holds
when all orbit maps are Lipschitz in time with a common Lipschitz constant.

\begin{definition}
We say that $\delta>0$ is a {\em geometrically separating constant} of $\varphi$ if whenever $x,y\in X$ and an increasing homeomorphism $s\colon\mathbb R\to\mathbb R$ satisfy
$s(0)=0$ and
    \[
    d\bigl(\varphi_t(x),\varphi_{s(t)}(y)\bigr)<\delta
    \qquad\text{for every }t\in\mathbb R,
    \]
    then
    \[
    y\in\varphi_{\mathbb R}(x).
    \]
    We say that $\varphi$ is \emph{geometrically separating} \cite{a} if it has a geometrically separating constant.
\end{definition}

The following lemma will be used to prove Theorem \ref{thmB}.

\begin{lemma}\label{lem:window}
Let $\varphi$ a uniformly $C_0$ flow of a metric space and
$\delta$ be a geometrically separating constant of $\varphi$.  Suppose that $\eta,\varepsilon>0$ satisfy
\begin{equation}\label{eq:scale}
 \varepsilon+\omega(\eta)<\delta.
\end{equation}
Let $T>0$, and let $\Gamma$ be a collection of distinct periodic orbits whose
least periods belong to an interval $I\subset(0,T]$ of length at most $\eta$.
If one point $x_\gamma$ is chosen from each $\gamma\in\Gamma$, then
$\{x_\gamma:\gamma\in\Gamma\}$ is $(T,\varepsilon)$-separated.
\end{lemma}

\begin{proof}
Suppose otherwise. Then there are distinct $\gamma,\gamma'\in\Gamma$, with
chosen points $x\in\gamma$ and $y\in\gamma'$, such that
\begin{equation}\label{eq:notsep}
 d(\varphi_u(x),\varphi_u(y))<\varepsilon
 \qquad(0\le u\le T).
\end{equation}
Write $p=\pi(\gamma)$ and $q=\pi(\gamma')$.  Since $p,q\in I$, we have
$|p-q|\le\eta$.

Define $s\colon\R\to\R$ by mapping each interval $[np,(n+1)p]$ linearly
onto $[nq,(n+1)q]$: for $n\in\mathbb Z$ and $0\le u\le p$, put
\[
 s(np+u)=nq+\frac{q}{p}u.
\]
This is an increasing homeomorphism and $s(0)=0$.  Periodicity gives
\[
 \varphi_{np+u}(x)=\varphi_u(x),
 \qquad
 \varphi_{s(np+u)}(y)=\varphi_{(q/p)u}(y).
\]
Moreover,
\[
 \left|\frac{q}{p}u-u\right|
 \le |q-p|\le\eta.
\]
Since $p\le T$, \eqref{eq:notsep}, the triangle inequality, and the definition
of $\omega$ yield
\begin{align*}
 d\bigl(\varphi_{np+u}(x),\varphi_{s(np+u)}(y)\bigr)
 &\le d(\varphi_u(x),\varphi_u(y))
      +d\bigl(\varphi_u(y),\varphi_{(q/p)u}(y)\bigr)\\
 &<\varepsilon+\omega(\eta)<\delta.
\end{align*}
This holds for every $n\in\mathbb Z$ and $u\in[0,p]$, hence for every real
time.  The geometric separating property implies $y\in\varphi_{\R}(x)$, contradicting
that $\gamma$ and $\gamma'$ are distinct periodic orbits.
\end{proof}

\begin{definition}
A set $F\subset X$ is called $(T,\varepsilon)$-spanning for $K$ if,
for every $x\in K$, there exists $y\in F$ such that
\[
d_T^\varphi(x,y)\leq\varepsilon.
\]
\end{definition}

We denote by
\[
s_T^\varphi(K,\varepsilon)\quad\mbox{ and }\quad r_T^\varphi(K,\varepsilon)
\]
the maximal cardinality of a $(T,\varepsilon)$-separated set for
$K$ and
the minimal cardinality of an $(T,\varepsilon)$-spanning set for $K$ respectively.

Since $K$ is compact, these numbers are finite.
The usual maximal-separated-set argument gives
\[
r_T^\varphi(K,\varepsilon)
\leq
s_T^\varphi(K,\varepsilon)
\leq
r_T^\varphi\left(K,\frac{\varepsilon}{2}\right).
\]
Consequently, separated and spanning sets define the same entropy.

\begin{definition}
\label{UC}
The upper capacity entropy of the flow $\varphi$ on the compact set
$K$ is
\[
\hE(\varphi,K)
:=
\lim_{\varepsilon\to0}
\limsup_{T\to\infty}
\frac1T\log s_T^\varphi(K,\varepsilon).
\]
Equivalently,
\[
\hE(\varphi,K)
=
\lim_{\varepsilon\to0}
\limsup_{T\to\infty}
\frac1T\log r_T^\varphi(K,\varepsilon).
\]
\end{definition}

This definition is the flow-version of the one in \cite{b1}.
The term "upper capacity" was coined in \cite{fh}.
The set $K$ is not required to be invariant under the flow.

\subsection{Comparison with the time-one map}

Let $f:X\to X$ be a continuous map.
For $n\in\mathbb{N}$, define the discrete-time dynamical metric
\[
d_n^f(x,y)
:=
\max_{0\leq j<n}
d\bigl(f^j(x),f^j(y)\bigr).
\]
Let
\[
s_n^f(K,\varepsilon)
\qquad\text{and}\qquad
r_n^f(K,\varepsilon)
\]
denote, respectively, the separated and spanning numbers associated
with the metric $d_n^f$ where $K\subset X$.
If $K$ is compact, one has
$$
\lim_{\varepsilon\to0}\limsup_{n\to\infty}\frac{1}n\log s_n^f(K,\varepsilon)=\lim_{\varepsilon\to0}\limsup_{n\to\infty}\frac{1}n\log r_n^f(K,\varepsilon).
$$
The common value is denoted by $\huc(f,K)$.

Unless otherwise specified we set
\[
f:=\varphi_1.
\]
The following lemma is standard and we prove it here for completeness.

\begin{lemma}
\label{l1}
For every compact set $K\subset X$,
\[
\hE(\varphi,K)
=
\huc(f,K).
\]
\end{lemma}

\begin{proof}
We give the proof using spanning sets.
Since the integer times
\[
0,1,\ldots,n-1
\]
belong to the interval $[0,n]$, one has
\[
d_n^f(x,y)\leq d_n^\varphi(x,y).
\]
Therefore every $(n,\varepsilon)$-spanning set for the flow is also
$(n,\varepsilon)$-spanning for the time-one map. Hence
\[
r_n^f(K,\varepsilon)
\leq
r_n^\varphi(K,\varepsilon).
\]
It follows that
\[
\lim_{\varepsilon\to0}
\limsup_{n\to\infty}
\frac1n\log r_n^f(K,\varepsilon)
\leq
\huc(\varphi,K).
\]

We now prove the reverse inequality. Fix $\varepsilon>0$. Since the
map
\[
[0,1]\times X\longrightarrow X,
\qquad
(t,x)\longmapsto\varphi_t(x),
\]
is continuous on a compact space, it is uniformly continuous.
Therefore there exists $\delta=\delta(\varepsilon)>0$ such that
\[
d(x,y)<\delta
\quad\Longrightarrow\quad
d\bigl(\varphi_u(x),\varphi_u(y)\bigr)<\varepsilon
\]
for every $u\in[0,1]$. We may choose
\[
0<\delta(\varepsilon)\leq\varepsilon,
\]
so that $\delta(\varepsilon)\to0$ as $\varepsilon\to0$.

Suppose that
\[
d_{n+1}^f(x,y)<\delta.
\]
Let $t\in[0,n]$. Write
\[
t=j+u,
\]
where
\[
j\in\{0,\ldots,n-1\},
\qquad
u\in[0,1].
\]
Then
\[
\varphi_t(x)=\varphi_u(f^j(x))
\]
and
\[
\varphi_t(y)=\varphi_u(f^j(y)).
\]
Since
\[
d\bigl(f^j(x),f^j(y)\bigr)<\delta,
\]
uniform continuity gives
\[
d\bigl(\varphi_t(x),\varphi_t(y)\bigr)<\varepsilon.
\]
Taking the maximum over $t\in[0,n]$, we obtain
\[
d_n^\varphi(x,y)<\varepsilon.
\]

Thus every $(n+1,\delta)$-spanning set for $K$ with respect to the
time-one map is $(n,\varepsilon)$-spanning for $K$ with respect to the
flow. Consequently,
\[
r_n^\varphi(K,\varepsilon)
\leq
r_{n+1}^f(K,\delta).
\]
Therefore
\[
\begin{aligned}
\limsup_{n\to\infty}
\frac1n\log r_n^\varphi(K,\varepsilon)
&\leq
\limsup_{n\to\infty}
\frac1n\log r_{n+1}^f(K,\delta)\\
&=
\limsup_{n\to\infty}
\frac1n\log r_n^f(K,\delta).
\end{aligned}
\]
Letting $\varepsilon\to0$, and hence
$\delta(\varepsilon)\to0$, gives
\[
\hE(\varphi,K)
\leq
\lim_{\delta\to0}
\limsup_{n\to\infty}
\frac1n\log r_n^f(K,\delta).
\]
Together with the opposite inequality, this proves
\[
\hE(\varphi,K)
=
\lim_{\varepsilon\to0}
\limsup_{n\to\infty}
\frac1n\log r_n^f(K,\varepsilon).
\]

Finally, the standard inequalities
\[
r_n^f(K,\varepsilon)
\leq
s_n^f(K,\varepsilon)
\leq
r_n^f\left(K,\frac{\varepsilon}{2}\right)
\]
show that the separated-set formula gives the same entropy.
\end{proof}

\begin{remark}
If $X$ is compact we can take $K=X$ so Lemma~\ref{l1} gives
\[
e(\varphi)=\hE(\varphi,X)=\hE(\phi).
\]
More generally, Lemma~\ref{l1} allows us to use
the time-one map in the auxiliary Bowen-entropy argument while retaining
the continuous-time definition of $\hE(\varphi,K)$ in the statement of
the main theorem.
\end{remark}

\begin{lemma}[Monotonicity]
\label{l2}
If $K_1\subset K_2$ are compact subsets of $X$, then
\[
\hE(\varphi,K_1)\leq\hE(\varphi,K_2).
\]
In particular, for every compact set $K\subset X$,
\[
\hE(\varphi,K)\leq e(\varphi).
\]
\end{lemma}

\begin{proof}
The result follows by Lemma \ref{l1} and Proposition 2.1 in \cite{fh} applied to $f=\varphi_1$.
\end{proof}

The set $\Sing(\varphi)$ is closed and invariant under the flow.
Consequently, $X^*$ is open and invariant.
Since $X$ is a compact metric space, it is Polish. Every open subset of
a Polish space is Borel and hence analytic. We shall use this fact when
applying the compact approximation theorem \cite{fh}.

Let $\cM_f(X)$ denote the set of $f$-invariant Borel probability
measures on $X$, and let $\cM_f^e(X)$ denote the subset of ergodic
measures. For $\mu\in\cM_f(X)$, let $h_\mu(f)$ denote the
measure-theoretic entropy \cite{w}.

\begin{lemma}\label{l3}
If $\mu\in\cM_f^e(X)$, then
$
\mu(\Sing(\varphi))\in\{0,1\}.
$
\end{lemma}

\begin{proof}
The set $\Sing(\varphi)$ is invariant under $f$. The conclusion therefore
follows from the ergodicity of $\mu$.
\end{proof}

\begin{lemma}
\label{l4}
If $\mu\in\cM_f(X)$ satisfies
\[
\mu(\Sing(\varphi))=1,
\]
then
\[
h_\mu(f)=0.
\]
\end{lemma}

\begin{proof}
For every $x\in\Sing(\varphi)$,
\[
f(x)=\varphi_1(x)=x.
\]
Thus $f$ agrees with the identity map on a set of full $\mu$-measure.

Let $\mathcal{P}$ be a finite measurable partition of $X$. Modulo sets
of $\mu$-measure zero,
\[
\bigvee_{j=0}^{n-1}f^{-j}\mathcal{P}
=
\mathcal{P}.
\]
Therefore
\[
h_\mu(f,\mathcal{P})
=
\lim_{n\to\infty}
\frac{1}{n}
H_\mu\left(
\bigvee_{j=0}^{n-1}f^{-j}\mathcal{P}
\right)
=
\lim_{n\to\infty}\frac{1}{n}H_\mu(\mathcal{P})
=
0.
\]
Taking the supremum over all finite measurable partitions gives
$h_\mu(f)=0$.
\end{proof}

\section{An Auxiliary Bowen-Entropy Argument}
\label{sec4}

\subsection{Bowen entropy of a subset}

We recall the Carath\'eodory-style definition of entropy by Bowen \cite{b}.
For $n\in\mathbb{N}$, $x\in X$, and $\varepsilon>0$, let
\[
B_n(x,\varepsilon)
:=
\left\{
y\in X:
d_n(x,y)<\varepsilon
\right\}.
\]
Recall
\[
d_n(x,y)
=
\max_{0\leq j<n}
d\bigl(f^j(x),f^j(y)\bigr).
\]

Let $Z\subset X$, $s\geq0$, $N\in\mathbb{N}$, and
$\varepsilon>0$. Define
\[
\mathcal{M}_{N,\varepsilon}^{s}(f,Z)
:=
\inf
\left\{
\sum_{i}e^{-s n_i}:
Z\subset
\bigcup_i B_{n_i}(x_i,\varepsilon),
\quad
n_i\geq N
\right\}.
\]
The infimum is taken over all finite or countable families of Bowen
balls
\[
\left\{
B_{n_i}(x_i,\varepsilon)
\right\}_i
\]
covering $Z$, with $x_i\in X$ and $n_i\geq N$.

Since
\[
\mathcal{M}_{N,\varepsilon}^{s}(f,Z)
\]
is nondecreasing as $N\to\infty$, the limit
\[
\mathcal{M}_{\varepsilon}^{s}(f,Z)
:=
\lim_{N\to\infty}
\mathcal{M}_{N,\varepsilon}^{s}(f,Z)
\]
exists.

There is a critical value at which
$\mathcal{M}_{\varepsilon}^{s}(f,Z)$ changes from $+\infty$ to $0$.
Define
\[
h_{\mathrm{top}}^{B}(f,Z,\varepsilon)
:=
\inf
\left\{
s\geq0:
\mathcal{M}_{\varepsilon}^{s}(f,Z)=0
\right\}.
\]
Equivalently,
\[
h_{\mathrm{top}}^{B}(f,Z,\varepsilon)
=
\sup
\left\{
s\geq0:
\mathcal{M}_{\varepsilon}^{s}(f,Z)=+\infty
\right\}.
\]

The Bowen topological entropy of $f$ on $Z$ is
\[
\hB(f,Z)
:=
\lim_{\varepsilon\to0}
h_{\mathrm{top}}^{B}(f,Z,\varepsilon).
\]
We use the convention
\[
\hB(f,\varnothing)=0.
\]

We require five standard properties.\\

\noindent
{\bf (P1)}
Bowen entropy is monotone:
\[
Z_1\subset Z_2
\quad\Longrightarrow\quad
\hB(f,Z_1)\leq\hB(f,Z_2).
\]
Moreover,
\[
\hB(f,X)=\htop(f).
\]
See Proposition 2.1 in \cite{fh}.\\

\noindent
{\bf (P2)} We have the following lemma.

\begin{lemma}[Entropy of a full-measure set]
\label{l5}
Let $f:X\to X$ be continuous on a compact metric space, and let
$\mu\in\mathcal{M}_f^e(X)$. If $Z\subset X$ is Borel and
\[
\mu(Z)=1,
\]
then
\[
h_\mu(f)\leq h_{\mathrm{top}}^B(f,Z).
\]
\end{lemma}

\begin{proof}
Let
\[
0<s<h_\mu(f).
\]
By the Brin--Katok local entropy theorem \cite{bk}, for $\mu$-almost every
$x\in X$,
\[
\lim_{\varepsilon\to0}
\liminf_{n\to\infty}
-\frac1n\log\mu\bigl(B_n(x,\varepsilon)\bigr)
=
h_\mu(f).
\]
Consequently, for sufficiently small $\varepsilon>0$, the set
\[
A
:=
\left\{
x\in Z:
\liminf_{n\to\infty}
-\frac1n\log\mu\bigl(B_n(x,2\varepsilon)\bigr)>s
\right\}
\]
has positive $\mu$-measure.

For $N\in\mathbb{N}$, define
\[
A_N
:=
\left\{
x\in A:
\mu\bigl(B_n(x,2\varepsilon)\bigr)
\leq e^{-sn}
\text{ for every }n\geq N
\right\}.
\]
Since
\[
A=\bigcup_{N=1}^{\infty}A_N,
\]
there exists $N$ such that
\[
\mu(A_N)>0.
\]

Consider any finite or countable cover
\[
A_N\subset\bigcup_i B_{n_i}(y_i,\varepsilon),
\qquad n_i\geq N.
\]
We may discard every ball that does not meet $A_N$. For each remaining
ball, choose
\[
x_i\in A_N\cap B_{n_i}(y_i,\varepsilon).
\]
The triangle inequality for the Bowen metric gives
\[
B_{n_i}(y_i,\varepsilon)
\subset
B_{n_i}(x_i,2\varepsilon).
\]
Since $x_i\in A_N$ and $n_i\geq N$,
\[
\mu\bigl(B_{n_i}(y_i,\varepsilon)\bigr)
\leq
\mu\bigl(B_{n_i}(x_i,2\varepsilon)\bigr)
\leq e^{-s n_i}.
\]
Therefore
\[
\mu(A_N)
\leq
\sum_i\mu\bigl(B_{n_i}(y_i,\varepsilon)\bigr)
\leq
\sum_i e^{-s n_i}.
\]
Taking the infimum over all such covers yields
\[
\mathcal{M}_{N,\varepsilon}^{s}(f,A_N)
\geq\mu(A_N)>0.
\]
It follows that
\[
h_{\mathrm{top}}^B(f,A_N,\varepsilon)\geq s.
\]
Since $A_N\subset Z$, monotonicity gives
\[
h_{\mathrm{top}}^B(f,Z)\geq s.
\]
As this holds for every $s<h_\mu(f)$, we conclude that
\[
h_{\mathrm{top}}^B(f,Z)\geq h_\mu(f).
\]
This completes the proof.
\end{proof}

\noindent
{\bf (P3)}
If $\htop(f)<\infty$ and $Z$ is analytic, then the compact
approximation theorem of Feng and Huang (Item (ii) of Theorem 1.2 in \cite{fh}) gives
\[
\hB(f,Z)
=
\sup_{\substack{K\subset Z\\K\ \mathrm{compact}}}
\hB(f,K).
\]

\noindent
{\bf (P4)}
The following is the standard comparison between the Carath\'eodory-type Bowen
entropy and the separated/spanning entropy of a set.

\begin{lemma}
\label{l6}
For every compact set $K\subset X$,
\[
\hB(f,K)\leq\hE(\varphi,K).
\]
\end{lemma}

\begin{proof}
By Item (iii) of Proposition 2.1 in \cite{fh} one has $h^B_{\mathrm{top}}(f,K)\leq \huc(f,K)$.
Recalling $f=\varphi_1$ we get the result from Lemma~\ref{l1}.
\end{proof}

\noindent
{\bf (P5)}
We have the following result.

\begin{lemma}
\label{l7}
The Bowen entropy of the nonsingular set equals the topological entropy
of the time-one map:
\[
\hB(f,X^*)=\htop(f).
\]
\end{lemma}

\begin{proof}
Since $X^*\subset X$, monotonicity gives
\[
\hB(f,X^*)\leq\hB(f,X)\overset{\mbox{\bf (P1)}}{=}\htop(f).
\]

For the reverse inequality, let $\mu\in\cM_f^e(X)$. By
Lemma~\ref{l3}, either
\[
\mu(\Sing(\varphi))=1
\quad\mbox{ or }\quad
\mu(X^*)=1.
\]

In the first case, Lemma~\ref{l4} gives
\[
h_\mu(f)=0\leq\hB(f,X^*).
\]
In the second case, $\mu(X^*)=1$. Therefore,
Lemma~\ref{l5} gives
\[
h_\mu(f)\leq\hB(f,X^*).
\]
Thus
\[
h_\mu(f)\leq\hB(f,X^*)
\]
for every $\mu\in\cM_f^e(X)$.

By the classical variational principle \cite{w},
\[
\htop(f)
=
\sup_{\mu\in\cM_f^e(X)}h_\mu(f).
\]
It follows that
\[
\htop(f)\leq\hB(f,X^*).
\]
Combining the two inequalities proves
\[
\hB(f,X^*)=\htop(f).
\]
\end{proof}

\section{Proof of the theorems}
\label{sec5}

\begin{proof}[Proof of Theorem \ref{thmA}]
We first prove the upper bound. For every compact set $K\subset X^*$,
Lemma~\ref{l2} gives
\[
\hE(\varphi,K)
\leq
\hE(\varphi,X)
=
e(\varphi).
\]
Therefore
\[
\sup_{\substack{K\subset X^*\\K\ \mathrm{compact}}}
\hE(\varphi,K)
\leq
e(\varphi).
\]

Conversely, 
set $f=\varphi_1$. By Lemma~\ref{l7},
\[
e(\varphi)=\htop(f)=\hB(f,X^*).
\]
Since $X^*$ is open, it is analytic. Since $\htop(f)=e(\varphi)<\infty$, the Feng--Huang compact
approximation theorem {\bf (P3)} therefore gives
\[
\hB(f,X^*)
=
\sup_{\substack{K\subset X^*\\K\ \mathrm{compact}}}
\hB(f,K).
\]
By Lemma~\ref{l6},
\[
\hB(f,K)\leq\hE(\varphi,K)
\]
for every compact $K\subset X^*$. Consequently,
\[
e(\varphi)=\hB(f,X^*)=
\sup_{\substack{K\subset X^*\\K\ \mathrm{compact}}}
\hB(f,K)\leq
\sup_{\substack{K\subset X^*\\K\ \mathrm{compact}}}
\hE(\varphi,K)
\]
completing the proof.
\end{proof}

\begin{proof}[Proof of Theorem \ref{thmB}]
Let $K\subset X$ be compact and meet every regular orbit.  Fix a
geometrically separating constant $\delta>0$.  Since $\varphi$ is uniformly $C_0$, choose
$\eta>0$ such that $\omega(\eta)<\delta/2$, and set
$\varepsilon=\delta/4$.  Then \eqref{eq:scale} holds.

For $T>0$, partition $(0,T]$ into
\[
 N_T=\left\lceil\frac{T}{\eta}\right\rceil
\]
half-open intervals of length at most $\eta$.  For each interval, select one
point in $K$ from every periodic orbit whose least period belongs to that
interval.  Lemma~\ref{lem:window} shows that the selected points form a
$(T,\varepsilon)$-separated subset of $K$.  Therefore, the number of periodic
orbits in each interval is at most $s_T^\phi(K,\varepsilon)$, and
\begin{equation}\label{eq:count}
 \nu(T)\le N_T s_T^\phi(K,\varepsilon)<\infty.
\end{equation}
Taking logarithms, dividing by $T$, and passing to the upper limit gives
\begin{align*}
 \limsup_{T\to\infty}\frac1T\log\nu(T)
 &\le \limsup_{T\to\infty}\frac1T\log N_T
      +\limsup_{T\to\infty}\frac1T\log s_T^\phi(K,\varepsilon).\\
 &\le \hE(\phi,K)\\
& \le \hE(\phi),
\end{align*}
because $T^{-1}\log N_T\to0$.  This completes the proof.
\end{proof}

\section{Applications}
\label{sec8}

\noindent
In this section, we give applications of our results according to the following subsections.

\subsection{Two corollaries}

We have two corollaries.

\begin{corollary}
Under the assumptions of Theorem~\ref{thmA}, for every
$\varepsilon>0$ there exists a compact set
\[
K_\varepsilon\subset X^*
\]
such that
\[
\hE(\varphi,K_\varepsilon)
>
e(\varphi)-\varepsilon.
\]
\end{corollary}

\begin{proof}
This follows directly from the definition of the supremum in
Theorem~\ref{thmA}.
\end{proof}

\begin{corollary}
If $\varphi$ is a differentiable flow of a compact manifold,
then
\[
e(\varphi)
=
\sup_{\substack{K\subset X^*\\K\ \mathrm{compact}}}
\hE(\varphi,K).
\]
\end{corollary}

\begin{proof}
Every diﬀerentiable flow on a compact manifold has finite topological entropy (see Theorem 7.15 in \cite{w}). Then, the result follows from Theorem \ref{thmA}.
\end{proof}

\subsection{Geometric separation outside singularities}
Let $\phi$ be a flow of a compact metric space $X$.
A subset $\Lambda\subset X$ is {\em dynamically isolated} if there is a neighborhood $U$ of $\Lambda$ such that
\begin{equation}
\label{beixo}
\Lambda=\bigcap_{t\in\mathbb{R}}\phi_t(U).
\end{equation}

By using Theorem \ref{thmB} we obtain the following corollary.

\begin{corollary}
\label{thm:A}
Let $\phi$ be a flow on a compact metric space $X$. Suppose that
$\phi$ is geometrically separating on
\[
X^*=X\setminus\Sing(\phi)
\]
and that $\Sing(\phi)$ is dynamically isolated. Then $\phi$ satisfies
the Bowen--Walters inequality.
\end{corollary}

\begin{proof}
Since $X$ is compact, the restricted flow $\varphi=\phi|_{X^*}$ is uniformly
$C_0$. Choose $U$ satisfying \eqref{beixo} for $\Lambda=\mathrm{Sing}(\phi)$. Then,
\[
K:=X\setminus U
\]
is a compact subset of $X^*$ meeting every regular orbit. Indeed, if the
orbit of some $x\in X^*$ did not meet $K$, then
\[
\phi_{\mathbb R}(x)\subset U,
\]
and the dynamical isolation would imply that
$x\in\Sing(\phi)$, a contradiction. Thus $\varphi$ is dynamically
isolated at infinity.

Since $\phi$ is geometrically separating, so does $\varphi$.
Then, Theorem \ref{thmB} gives
\[
\limsup_{T\to\infty}\frac{1}{T}\log\nu(T)
\leq \hE(\phi)=\sup_{K\subset X^*\  \mathrm{compact}}\hE(\phi,K).
\]
For every compact set $K\subset X^*$,
\[
\hE(\phi,K)\leq e(\phi),
\]
because every separated subset of $K$ is also a separated subset of $X$.
Taking the supremum over all compact subsets $K\subset X^*$ proves
\[
\limsup_{T\to\infty}\frac{1}{T}\log\nu(T)
\leq e(\phi).
\]
This completes the proof.
\end{proof}

In particular, the Bowen-Walters inequality holds
for every geometrically separating flow with dynamically isolated singular set of a compact metric space.

\subsection{Singular suspensions}
We recall the differentiable construction of a singular suspension, following
the approach used by Komuro \cite{k}. Let $M$ be a compact manifold and let
$f\colon M\to M$ be a diffeomorphism. The suspension manifold of $f$ is
\[
M_f
:=
(M\times[0,1])/\!\sim,
\]
where
\[
(x,1)\sim(f(x),0)
\qquad\text{for every }x\in M.
\]
We denote the equivalence class of $(x,r)$ by $[x,r]$. The ordinary
suspension flow is
\[
\psi_t([x,r])=[x,r+t],
\]
where the equivalence relation is used when the second coordinate crosses
$0$ or $1$. Let $V$ be the smooth nonsingular vector field on $M_f$
generating $\psi$.

Fix a point
\[
\sigma=[a,r_0]\in M_f,
\qquad a\in M,\quad 0<r_0<1,
\]
and choose a smooth function
\[
\alpha\colon M_f\to[0,\infty)
\]
such that
\[
\alpha^{-1}(0)=\{\sigma\}.
\]
The vector field
\[
W:=\alpha V
\]
is complete because $M_f$ is compact. The flow $\phi$ generated by $W$
is called a \emph{singular suspension} of $f$, and $\alpha$ is called a
\emph{brake function}. Its singular set is
\[
\Sing(\phi)=\{\sigma\}.
\]

On $M_f\setminus\{\sigma\}$, the vector fields $V$ and $W$ are positively
collinear. Consequently, the ordinary suspension flow $\psi$ and the
singular suspension flow $\varphi$ have the same oriented orbit segments
away from $\sigma$, although they have different time parametrizations.
The ordinary suspension orbit through $\sigma$ is divided into regular
$\phi$-orbits which approach $\sigma$ asymptotically.

We also recall that a homeomorphism $f\colon M\to M$ is called
\emph{expansive} if there exists $c>0$ such that
\[
d\bigl(f^n(x),f^n(y)\bigr)<c
\qquad\text{for every }n\in\mathbb Z
\]
implies
\[
x=y.
\]
Such a number $c$ is called an \emph{expansivity constant} for $f$.

Let $\Lambda$ be a compact invariant set of a flow
$\phi$.
The restricted flow
$\phi|_\Lambda$ is called \emph{$k^*$-expansive} if, for every
$\varepsilon>0$, there exists $\delta>0$ such that, whenever
$x,y\in\Lambda$ and an increasing homeomorphism $s:\mathbb{R}\to\mathbb{R}$ fixing $0$ satisfy
\[
d\bigl(\phi_t(x),\phi_{s(t)}(y)\bigr)<\delta
\qquad\text{for every }t\in\mathbb R,
\]
there exists $t_0\in\mathbb R$ for which
\[
\phi_{s(t_0)}(y)
   \in
\phi_{[-\varepsilon,\varepsilon]}
       \bigl(\phi_{t_0}(x)\bigr).
\]
This definition is adapted to flows with singularities and was
introduced by Komuro \cite{k1} in his study of Lorenz attractors. In particular,
$k^*$-expansivity implies geometric separation: fixing any
$\varepsilon>0$ and taking the corresponding $\delta$, the preceding
conclusion implies
\[
y\in\phi_{\mathbb R}(x).
\]
Consequently, every $k^*$-expansive flow is geometrically separating,
and hence geometrically separating on its regular points.
We can now state the following consequence for singular suspensions of expansive diffeomorphisms.

\begin{corollary}
The singular suspension $\phi$ of an expansive homeomorphism of a compact manifold satisfies the Bowen-Walters inequality
\[
\limsup_{T\to\infty}\frac{1}{T}\log\nu(T)
\leq e(\phi).
\]
\end{corollary}

\begin{proof}
The singular suspension of every expansive homeomorphism is
$k^*$-expansive \cite{s} and hence geometrically separating on its regular part.
Moreover, its unique singularity is dynamically isolated. Indeed, one can
choose a compact global cross-section disjoint from the singularity, and
every regular orbit meets this cross-section in forward or backward time.
The result therefore follows from Corollary~\ref{thm:A}.
\end{proof}

\begin{remark}
The preceding corollary does not follow directly from the
Bowen--Walters inequality for the expansive homeomorphism defining the
singular suspension. Indeed, the latter estimates periodic points in
terms of their discrete periods, whereas $\nu(T)$ counts periodic
orbits of the suspension according to their flow periods. Since the
braking function vanishes at the singularity, these two notions of
period need not be uniformly comparable.

This corollary is complementary to the results of Rego and Romaña
\cite{rr}. They study the preservation of positive entropy under
singular suspension. In particular, they prove that every expansive
$C^1$ flow on a closed three-dimensional manifold which is a singular
suspension of a diffeomorphism has positive topological entropy.
Moreover, they prove that if the base diffeomorphism is expansive and
the braking function has finitely many zeros, then the corresponding
singular suspension is $k^*$-expansive. Consequently, such a flow is geometrically separating and our result yields the Bowen-Walters inequality for such flows.

Thus, in the three-dimensional setting considered in \cite{rr}, the
right-hand side in Bowen-Walters's is additionally known to be positive. Their results
and ours address complementary questions: Rego and Romaña establish
positivity of the entropy, while the present result bounds the
exponential growth of periodic orbits from above by that entropy.
\end{remark}

\subsection{Multi-singular hyperbolic sets}
\label{subsec:multisingular}

Multi-singular hyperbolicity was introduced by Bonatti and da Luz
\cite{bdl} as a framework for describing chain-recurrence classes
containing singularities of different indices. It extends the
three-dimensional theory of singular hyperbolicity and occurs naturally
in the study of star flows.

A recent theorem of Pacifico et al \cite{pyyz} states
that every multi-singular hyperbolic set is robustly $k^*$-expansive. Since $k^*$-expansivity implies geometric
separation, their theorem provides the following application of our
Bowen--Walters inequality.

\begin{corollary}[Multi-singular hyperbolic sets]
\label{cor:multisingular}
Let $\Lambda$ be a multi-singular hyperbolic set of a $C^1$ flow $\phi$ on a compact Riemannian manifold. Then
\[
\limsup_{T\to\infty}
\frac{1}{T}\log\nu_\Lambda(T)
\leq e(\phi|_\Lambda),
\]
where $\nu_\Lambda(T)$ denotes the number of periodic orbits contained
in $\Lambda$ whose period is at most $T$.
\end{corollary}

\begin{proof}
By the aforementioned result in
\cite{pyyz}, the restricted flow $\phi|_\Lambda$ is $k^*$-expansive. It follows that $\phi|_\Lambda$ is geometrically
separating and, in particular, geometrically separating on
$
\Lambda_{\mathrm{reg}}
:=
\Lambda\setminus\Sing(\phi).
$
Every singularity contained in a multisingular-hyperbolic set is hyperbolic. Since $\Lambda$ is compact, the set
$
\Lambda\cap\Sing(\phi)
$
is finite. Each of its elements is an isolated invariant set.
Consequently, their finite union is dynamically isolated for the
restricted flow $\phi|_\Lambda$.
Finally, compactness of $\Lambda$ and continuity of the flow imply that $\phi|_\Lambda$ is uniformly $C_0$. All the hypotheses of
Corollary~\ref{thm:A} are satisfied, and the claimed
inequality follows.
\end{proof}

\begin{remark}
This corollary extends the upper-bound part of
the periodic-orbit results in Theorem B of
\cite{pyyz} to arbitrary multisingular-hyperbolic sets. Indeed, their
sharp counting theorem assumes that $\Lambda$ is a positive-entropy
multisingular-hyperbolic chain-recurrence class whose periodic orbits
are homoclinically related. Under these assumptions, they prove
\[
\nu_\Lambda(T)\asymp
\frac{\exp\bigl(T e(\phi|_\Lambda)\bigr)}{T},
\]
together with equidistribution of the corresponding periodic-orbit
measures.

Our conclusion is weaker than this sharp asymptotic estimate, but it
requires neither chain recurrence, homoclinic relatedness, genericity,
nor positivity of the entropy. In particular, it applies to every
multisingular-hyperbolic set and implies that, when
$e(\phi|_\Lambda)=0$, its periodic orbits have at most subexponential
growth.
\end{remark}

The sectional-hyperbolic attractors \cite{mm} form an important
special case. We therefore obtain the following consequence.

\begin{corollary}[Sectional-hyperbolic attractors]
\label{cor:singular-hyperbolic}
Let $\Lambda$ be a sectional-hyperbolic attractor of a $C^1$ vector
field on a compact three-dimensional Riemannian manifold. Then
\[
\limsup_{T\to\infty}
\frac{1}{T}\log\nu_\Lambda(T)
\leq e(\phi|_\Lambda).
\]
In particular, this conclusion applies to geometric Lorenz
attractors \cite{gw}.
\end{corollary}

\begin{proof}
Every sectional-hyperbolic attractor is
multisingular hyperbolic. The result follows immediately from
Corollary~\ref{cor:multisingular}.
\end{proof}

\subsection{Asymptotic sectional-hyperbolicity}
Let $G$ be a $C^1$ vector field on a compact Riemannian manifold $M$,
let $\phi=(\phi_t)_{t\in\mathbb R}$ be its flow, and let $\Lambda$ be
a compact invariant partially hyperbolic set with splitting
\[
T_\Lambda M=E^s\oplus E^{cu}.
\]
Assume that every singularity of $G$ contained in $\Lambda$ is
hyperbolic. Following \cite{ms}, the set $\Lambda$
is called \emph{asymptotically sectional-hyperbolic}, or \emph{ASH},
if there exists $c>0$ such that, for every
\[
x\in
\Lambda\setminus
\bigcup_{\sigma\in\Sing(\phi)\cap\Lambda}W^s(\sigma)
\]
and every two-dimensional subspace $L_x\subset E^{cu}_x$, one has
\[
\limsup_{t\to+\infty}
\frac{1}{t}
\log
\left|
\det\left(D\phi_t(x)|_{L_x}\right)
\right|
\geq c.
\]
Thus, unlike sectional hyperbolicity, the ASH condition does not
require uniform sectional expansion at every positive time. It
requires only asymptotic sectional expansion and consequently provides
arbitrarily large hyperbolic times for points outside the stable
manifolds of the singularities.

Every sectional-hyperbolic set is ASH, but the converse does not hold.
The principal example is the Rovella attractor. San Mart\'{i}n and Vivas \cite{sv}
proved that the Rovella attractor is ASH \cite{sv1}, whereas it
is not sectional-hyperbolic because its singularity is of Rovella
type. General three-dimensional ASH attractors are known to be
rescaling expansive \cite{rv}, but rescaling expansivity alone does not imply
geometric separation. Nevertheless, Carrasco-Olivera and San Martín
proved directly that the Rovella attractor \cite{k} is $k^*$-expansive
\cite{csm}. Since $k^*$-expansivity implies geometric separation, the
Rovella attractor provides an application of our Bowen--Walters
inequality.

\section*{Declaration of competing interest}

\noindent
There is no competing interest.

\section*{Data availability}

\noindent
No data was used for the research described in the article.

% Insert the compact singular corollary, followed by the examples involving
% singular suspensions, sectional-hyperbolic attractors, and Rovella attractors.

\end{document}